\documentclass[12pt,psamsfonts]{amsart}
\usepackage{amssymb,eucal,epsfig}

\theoremstyle{plain}
\newtheorem*{Morley}{Morley's  Congruence \cite{Morl}}
\newtheorem{lemma}{Lemma}

\begin{document}
\title[Morley's other miracle]{Morley's other miracle: $\displaystyle 4^{p-1}\equiv\pm \left(\smallmatrix p-1\\
\frac{p-1}{2} \endsmallmatrix\right) 
\pmod {p^3}
$}

\author{Christian Aebi and Grant Cairns}

\address{Coll\`ege Calvin, Geneva, Switzerland 1211}
\email{christian.aebi@edu.ge.ch}
\address{Department of Mathematics, La Trobe University, Melbourne, Australia 3086}
\email{G.Cairns@latrobe.edu.au}

\maketitle

In geometry, {\em Morley's miracle} says that in every planar triangle the adjacent angle trisectors meet at the
vertices of an equilateral triangle. Frank Morley obtained this wonderful result in 1899, and to this day it continues to attract interest. There are now many known proofs; see the cut-the-knot web site \cite{CTK}. Perhaps the most celebrated ones are those due to Alain Connes \cite{Connes} and  John Conway (unpublished, yet accessible at \cite{CTK}). A proof in the same spirit as Connes'  was published earlier by Liang-shin Hahn \cite{Hahn}; see also \cite{Geiges}. Conway's proof is perhaps the simplest and nicest one; a somewhat longer proof having the same general approach was given by Coxeter \cite{Coxeter}, and attributed to Raoul Bricard; see also \cite{New,Wa}.

\centerline{\includegraphics[scale=.75]{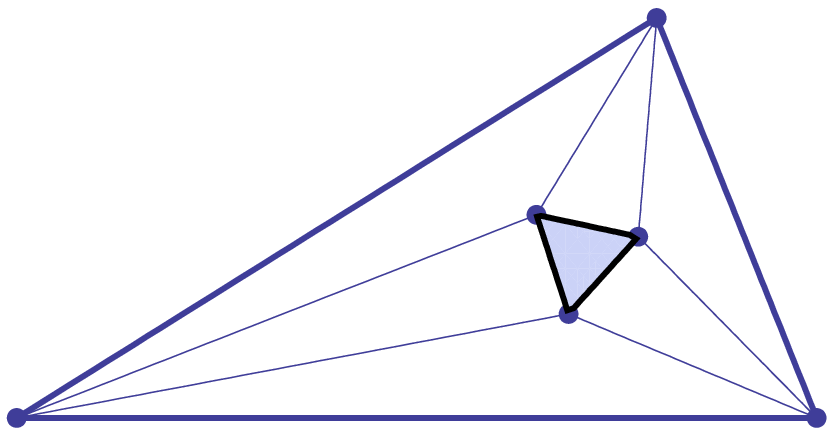}}
\bigskip

Morley's miracle was by no means his sole
surprising discovery. In number theory, he published the following result  in the Annals of Mathematics 1894/95.  

\begin{Morley} If $p$ is prime and 
$p>3$, then 
\[(-1)^{(p-1)/2}\cdot \left(\smallmatrix p-1\\
\frac{p-1}{2} \endsmallmatrix\right) \equiv2^{2p-2}
\pmod {p^3}.
\]
\end{Morley}

To appreciate the ``miraculous'' nature of this congruence, one first needs to compare it  with other congruences known at the time. Some famous ones for primes $p$ include:
\begin{itemize}
\item Fermat's little theorem:
$2^{p-1}\equiv 
1\pmod p$.

\item Wilson's theorem:  $(p-1)!\equiv -1 \pmod p$.

\item Lucas' theorem:
 If  $0\leq n,j<p$, then
$\left(\smallmatrix pm+n\\ pi+j\endsmallmatrix\right)\equiv \left(\smallmatrix m\\ i\endsmallmatrix\right)\left(\smallmatrix n\\ j\endsmallmatrix\right)\pmod {p}$.

\end{itemize}
The above three congruences are modulo $p$, while Morley's congruence is modulo $p^3$. The difference between mod $p^3$ and mod $p$ is analogous to having a result to  three significant figures, rather than just one significant figure.

The other striking aspect of Morley's congruence was the nature of his original proof, which  made an ingenious use of integration of
trigonometric sums. First  he used the Fourier series:
\begin{align*}
2^{2n}\cos^{2n+1}x&=\cos(2n+1)x+(2n+1)\cos(2n-1)x+\frac{(2n+1)2n}{1.2}\cos(2n-3)x\\&\quad+\dots+\frac{(2n+1)2n\dots(n+2)}{n!}\cos x.
\end{align*}
He integrated this term by  term and compared it with the following formula, which can be obtained by induction using integration by parts:
\begin{equation}\label{ind}
\int_0^{\frac12\pi}\cos^{2n+1}x dx=\frac{2n(2n-2)\dots2}{(2n+1)(2n-1)\dots3}.
\end{equation}
This established his result modulo $p^2$, where $p=2n+1$.  To obtain the result modulo $p^3$, Morley then used (\ref{ind}) again to integrate the following power series in $\cos(x)$,  known from``treatises on trigonometry'':
\begin{align*}
(-1)^{\frac{p-1}2}\cos px&=p\cos x-\frac{p(p^2-1^2)}{3!}\cos^3x+\frac{p(p^2-1^2)(p^2-3^2)}{5!}\cos^5x\\&\quad-\dots+(-1)^{\frac{p-1}2}2^{p-1}\cos^px.
\end{align*}

Subsequently, two alternate proofs were given that used the properties of Bernoulli numbers: the 1913
Royal Danish Academy of Sciences paper by Niels Nielsen \cite[p.~353]{Nielsen} and the 1938 Annals of Mathematics paper by Emma Lehmer
\cite[p.~360]{Lehmer}. 

The main aim of this note is to establish Morley's congruence by entirely elementary number theory arguments. The key to this approach is the following
basic congruence modulo $p$ that curiously, we have not seen in the literature.  

\begin{lemma}\label{L1} If $p$ is prime and 
$p>3$, then $\displaystyle
\sum_{\substack{0< i<j< p\\i\text{ odd},j\text{ even}}} \frac {1}{ij}\equiv 0\pmod {p}$.
\end{lemma}

Here, $\frac1{ij}$ denotes the multiplicative inverse of $ij$ modulo
$p$. Throughout this note, $p$ is a prime greater than 3 and by a slight abuse of notation, $\frac1i$ will denote the fraction $1/i$ or the multiplicative inverse of $i$ modulo
$p$ or modulo $p^2$, according to the context.

After we have established Morley's congruence, we will show in the final section that it can also be deduced from Granville's elegant proof of  Skula's conjecture \cite{Granville}.

\section*{Reduction of the Problem}

We will use the following well known facts \cite[Theorem 117]{HW}, that we prove for completeness.

\begin{lemma}\label{L2} {\rm(a)} $\sum_{i=1}^{\frac{p-1}2} \frac{1}{i^2}\equiv 0\pmod {p}$,
\qquad{\rm(b)} $\sum_{i=1}^{p-1}\frac{(-1)^i}{i}\equiv \sum_{i=1}^{\frac{p-1}2} \frac{1}{i}\pmod{p^2}$.
\end{lemma}
\begin{proof}(a) As  $\frac{1}{i^2} \equiv
\frac{1}{(p-i)^2} \pmod p$, one has
\[
2\sum_{i=1}^{\frac{p-1}2} \frac{1}{i^2}= \sum_{i=1}^{p} \frac{1}{i^2} = \sum_{i=1}^{p} i^2 =\frac{(p-1)p(2p-1)}6\equiv 0\pmod {p}.\]

(b) For all $0< i\leqslant {\frac{p-1}2}$, one has $i(p-i)+i^2\equiv -p(p-i)\pmod{p^2}$ and dividing by $i^2(p-i)$ gives $\frac{1}i+\frac1{p-i}\equiv
-\frac{p}{i^2}\pmod{p^2}$. Summing and using (a) gives $\sum_{i=1}^{p-1}\frac1{i}\equiv 0\pmod{p^2}$, which is known as Wolstenholme's theorem. Thus
\[
\sum_{i=1}^{p-1}\frac{(-1)^i}i\equiv 2\sum_{\substack{i=2\\i\text{ even}}}^{p-1}\frac1{i}\equiv \sum_{i=1}^{\frac{p-1}2}
\frac1{i}\pmod{p^2}.
\]\end{proof}

Turning to the terms in Morley's congruence, first note that
\[
\binom{p}{i}=\frac{{p}\cdot{(p-1)}\cdot{(p-2)} \cdots {(p-(i-1))}}{{i} \cdot {1} \cdot {2} \cdots
{(i-1)}}
\]
and so 
\begin{equation}\label{exp}
\binom{p}{i}=(-1)^{i-1}\cdot\frac{p}i \cdot \left(1-\frac{p}1\right) \cdot
\left(1-\frac{p}2 \right)\cdots \left(1-\frac{p}{i-1}\right).
\end{equation}
Thus
$\binom{p}{i}\equiv(-1)^{i}\cdot\left(-\frac{p}i+ p^2\cdot\sum_{j=1}^{i-1} \frac 1{ij}\right)\pmod {p^3}$
and so $2^p= 2+\sum_{i=1}^{p-1} \binom{p}{i}$ gives 
\[
2^{p-1}\equiv 1-\frac{p}2\cdot\sum_{i=1}^{{p-1}} \frac{(-1)^{i}}i+ \frac{p^2}2\cdot\sum_{0< j< i<p} \frac {(-1)^{i}}{ij}\pmod
{p^3}.
\]
Squaring, and using Lemma \ref{L2}(b), we have 
\begin{equation}\label{lh}
2^{2p-2}\equiv 1-p\cdot\sum_{i=1}^{\frac{p-1}2} \frac1i+ p^2\cdot\left(\frac14\left(\sum_{i=1}^{\frac{p-1}2}
\frac1i\right)^2+ \sum_{0< j< i<p} \frac{(-1)^{i}}{ij}\right)\pmod {p^3}.
\end{equation}
From (\ref{exp}) we also have
$(-1)^{i-1}\binom{p-1}{i-1}=(-1)^{i-1}\frac{i}{p}\binom{p}{i}= \left(1-\frac{p}1\right) \cdot
\left(1-\frac{p}2 \right)\cdots \left(1-\frac{p}{i-1}\right)$.
Taking $i=\frac{p+1}2$ gives 
$(-1)^{\frac{p-1}{2}}\cdot\binom{p-1}{\frac{p-1}{2}} \equiv
1-p\cdot\sum_{i=1}^{\frac{p-1}2} \frac{1}i+p^2\cdot \sum_{1\leq j<i\leq \frac{p-1}2}\frac{1}{i j}\pmod{p^3}$,
or equivalently, using Lemma \ref{L2}(a),
\begin{equation}\label{rh}(-1)^{\frac{p-1}{2}}\cdot\binom{p-1}{\frac{p-1}{2}} \equiv
1-p\cdot\sum_{i=1}^{\frac{p-1}2} \frac{1}i+\frac{p^2}2\cdot \left(\sum_{i=1}^{\frac{p-1}2} \frac1{i}\right)^2\pmod{p^3}.
\end{equation}
Comparing (\ref{lh}) and (\ref{rh}), we observe that Morley's congruence is therefore
valid mod $p^2$.  In order to obtain it mod $p^3$, it suffices to prove that
$\frac14\left(\sum_{i=1}^{\frac{p-1}2}
\frac1i\right)^2\equiv  \sum_{0< j< i<p} \frac{(-1)^{i}}{ij}\pmod {p}$,
or equivalently,
\begin{equation}\label{mo}
\left(\sum_{\substack{0< i< p\\i\text{ even}}} \frac1i\right)^2\equiv  \sum_{0< j< i<p} \frac{(-1)^{i}}{ij}\pmod {p}.
\end{equation}
The considerations so far have reduced Morley's congruence modulo $p^3$ to a congruence modulo $p$.

\section*{Completion of the Proof}

In the remainder of this note, all congruences are taken modulo $p$. First notice that as $\displaystyle\sum_{\substack{0< i< p\\i\text{ even}}} \frac {1}{i}= -\sum_{\substack{0< i< p\\i\text{ odd}}} \frac
{1}{i}$,  the left hand side of (\ref{mo}) is
\[
\left(\sum_{\substack{0< i< p\\i\text{ even}}} \frac {1}{i}\right)^2\equiv -\left(\sum_{\substack{0< i< p\\i\text{
odd}}} \frac {1}{i}\right)\left(\sum_{\substack{0< j< p\\j\text{ even}}} \frac {1}{j}\right)\equiv -\sum_{\substack{0<
j<i< p\\i\text{ odd},j\text{ even}}} \frac {1}{ij} -\sum_{\substack{0< i<j< p\\i\text{ odd},j\text{ even}}} \frac {1}{ij}.\]
On the other hand, 
\[
\sum_{\substack{0< j<i< p\\i,j\text{ odd}}} \frac {1}{ij}= \sum_{\substack{0< i<j< p\\i,j\text{ even}}} \frac
{1}{(p-i)(p-j)}\equiv \sum_{\substack{0< i<j< p\\i,j\text{ even}}} \frac
{1}{ij}
 \]
and so the right hand side of (\ref{mo}) is
\begin{align*}
\sum_{0< j< i<p} \frac{(-1)^{i}}{ij}&=\sum_{\substack{0< j<i< p\\i,j\text{ even}}}
\frac{1}{ij}-\sum_{\substack{0< j<i< p\\i,j\text{ odd}}}
\frac{1}{ij}-\sum_{\substack{0< j<i< p\\i\text{ odd}, j\text{ even}}}
\frac{1}{ij}+\sum_{\substack{0< j<i< p\\i\text{ even},j\text{ odd}}}
\frac{1}{ij}\\
&\equiv -\sum_{\substack{0< j<i< p\\i\text{ odd},j\text{ even}}} \frac {1}{ij} +
\sum_{\substack{0<i<j< p\\i\text{ odd},j\text{ even}}}
\frac{1}{ij}.
\end{align*}
Hence (\ref{mo}) follows from Lemma \ref{L1}, and so the proof of Lemma \ref{L1} is our final task.

\begin{proof}[Proof of Lemma \ref{L1}] We have
\begin{align*}
2\sum_{\substack{0< i<j< p\\i\text{ odd},j\text{ even}}} \frac {1}{ij}=
\sum_{\substack{0< i<j< p\\i\text{ odd},j\text{ even}}} &\frac {1}{ij}+\frac {1}{(j-i)j}
=\sum_{\substack{0< i<j< p\\i\text{ odd},j\text{
even}}} \frac
{1}{i(j-i)}=\sum_{\substack{0< i,k< p\\i+k<p\\i,k\text{ odd}}} \frac {1}{ik}\\
\equiv &\sum_{\substack{0< i<j< p\\i\text{ odd},j\text{
even}}} \frac {1}{i(p-j)}\equiv -\sum_{\substack{0< i<j< p\\i\text{ odd},j\text{
even}}} \frac
{1}{ij}
\end{align*}
which gives the required result, as $p>3$.\end{proof}

\section*{The connection with Skula's conjecture}\label{proof1} 

Consider the Fermat quotient $q=\frac{2^{p-1}-1}p$, and note that
\begin{equation}\label{bas}
2^{2p-2}=1+2qp+q^2p^2.
\end{equation}
Adopting the notation of \cite{Granville}, set 
\[
q(x)=\frac{x^p-(x-1)^p-1}p,\qquad g(x)=\sum_{i=1}^{p-1}\frac{x^i}i,\qquad G(x)=\sum_{i=1}^{p-1}\frac{x^i}{i^2}.\]
Note that $q=q(2)/2$. The following remarkable identity was established in \cite{Granville}:
\begin{equation}\label{Gr}
-G(x)\equiv \frac1p (q(x)+g(1-x))\pmod p,
\end{equation}
from which Granville deduced Skula's conjecture:
$q^2\equiv -G(2)\pmod p$.
From (\ref{Gr}), 
\[
2q\equiv -g(-1)-G(2)p\equiv -g(-1)+q^2p\pmod {p^2}.\]
Hence, substituting in (\ref{bas}), we obtain
\begin{equation}\label{rq}
2^{2p-2}=1+2qp+q^2p^2\equiv  1-g(-1)p+\frac12g(-1)^2p^2\pmod {p^3}.
\end{equation}
From Lemma \ref{L2}(b),  $g(-1)\equiv \sum_{i=1}^{\frac{p-1}2}
\frac1{i}\pmod{p^2}$, and so from (\ref{rq})
\[
2^{2p-2}\equiv 1-\left(\sum_{i=1}^{\frac{p-1}2} \frac1{i}\right)p +\frac12\left(\sum_{i=1}^{\frac{p-1}2} \frac1{i}\right)^2p^2 \pmod {p^3}.
\]
Together with (\ref{rh}), this gives Morley's congruence once again.

\vskip1cm
\noindent{\bf Acknowledgements.} We would like to thank the anonymous referees who proposed both constructive and illuminating remarks.

\bibliographystyle{amsplain}
\bibliography{morleyfinal}
\end{document}